\documentclass{amsproc}
\usepackage{amsfonts}
\usepackage{amssymb}
\usepackage{amsmath}
\usepackage{graphicx}
\newtheorem{theorem}{\textbf{Theorem}}

\newtheorem{corollary}{\textbf{Corollary}}

\newtheorem{definition}{\textbf{Definition}}
\newtheorem{example}{\textbf{Example}}
\newtheorem{lemma}{\textbf{Lemma}}
\newtheorem{proposition}{\textbf{Proposition}}
\newtheorem{remark}{\textbf{Remark}}

\begin{document}

\title[Convexity and some geometric properties]{Convexity and some geometric properties}

\author[J. X. Cruz Neto$^{1}$, \'Italo Melo$^{2}$ and Paulo Sousa$^{3}$]{J. X. Cruz Neto$^{1}$, \'Italo Melo$^{2}$ and Paulo Sousa$^{3}$}
%\author[J.X. Cruz Neto]{J.X. Cruz Neto$^{1}$}
%\author[\'Italo Melo]{\'Italo Melo$^{2}$}
%\author[Paulo Sousa]{Paulo Sousa$^{3}$}

%\address{Universidade Federal do Piau\'{i}, Departamento de
%Matem\'{a}tica, 64049-550, Ininga - Teresina - PI, Brazil\\
%$^{1}$email: jxavier@ufpi.edu.br\\
%$^{2}$email: italodowell@ufpi.edu.br\\
%$^{3}$email: paulosousa@ufpi.edu.br}

\thanks{The first author is supported by CNPq/Brazil. The third author is 
supported by CNPq/Brazil and PROCAD/CAPES/Brazil.}

%\thanks{Corresponding Author: J.X. Cruz Neto; E-mail: jxavier@ufpi.edu.br;
%Telephone/Fax: +55(86)3237-1609.}

\begin{abstract}

The main goal of this paper is to present results of existence and non-existence of convex functions on Riemannian manifolds and, in the case of the existence, we associate such functions to the geo\-metry of the manifold. Precisely, we prove that the conservativity of the geodesic flow on a Riemannain mani\-fold with infinite volume is an obstruction to the existence of convex functions. Next, we present a geometric condition that ensures the existence of (strictly) convex functions on a particular class of complete non-compact manifolds, and, we use this fact to construct a manifold whose sectional curvature assumes any real value greater than a negative constant and admits a strictly convex function. In the last result we relate the geometry of a Riemannian manifold of positive sectional curvature with the set of minimum points of a convex function defined on the manifold. 

\medskip

\noindent\textit{$\bf{Key \ words:}$} Convex function, Geodesic flow, Conformal fields, Soul of a manifold.

\vspace{.2cm}

%\noindent\textit{$\bf{MSC:}$} 37A05, 52A38, 37C10.

\medskip

\begin{center}
Universidade Federal do Piau\'{i}, Departamento de
Matem\'{a}tica, 64049-550, Ininga - Teresina - PI, Brazil\\
$^{1}$email: jxavier@ufpi.edu.br\\
$^{2}$email: italodowell@ufpi.edu.br\\
$^{3}$email: paulosousa@ufpi.edu.br
\end{center}

\end{abstract}

\maketitle
\section{Introduction}

The concept of convexity plays a very important role in optimization theory, firstly because many objective functions are convex in a sufficiently small neighborhood of a local minimum point, and secondly because one can establish the convergence of numerical methods to estimate minimum points for convex functions.

Convex functions occur abundantly, with many structural implications on Riemannian manifolds and form an important link between modern analysis and geometry. The existence of convex functions implies restrictions on the geometry or topology of a complete non-compact Riemannian manifold. For example: Bishop and O'Neill \cite{BO} proved that there is no non-trivial smooth convex function on a complete Riemannian manifold with finite volume; later, Yau \cite{Y} generalized the result of Bishop and O'Neill, proving that there is no non-trivial continuous convex function on a complete manifold with finite volu\-me; Shio\-hama \cite{S} proved a result relating the existence of strictly convex functions to the topology of the Riemannian manifold, actually he proved that if a complete Riemannian manifold admits a strictly convex function, then the manifold has at most two ends. 

In this paper we obtain a geometric-topological restriction to the existence of non-trivial convex functions on complete non-compact Riemannian manifolds, more specifically, we prove that the conservativity of the geodesic flow on a Riemannian manifold with infinite volume implies that all convex functions on the manifold are constant, this fact generalizes in a certain sense the result proved by Yau \cite{Y}. 

%%%%%%%%%%%%%%%%%%%%%%%%%%%%%%%%%%%%%%%%%%%%%

The technique used in the proof of convergence of many methods of minimizing convex functions on Riemannian manifolds are well-posed when the sectional curvature does not change sign. The works of Cruz Neto et al. \cite{CNLO} and Ferreira-Oliveira \cite{FO} are pioneers in the class of works involving methods of minimizing convex functions on Riemannian manifolds of non-negative sectional curvature. In 2002, Ferreira and Oliveira \cite{FO1} established proximal point method in Hadamard manifolds (non positive sectional curvature). 

A first attempt to escape the sign limiting of the sectional curvature, as far as we know, was made by Wang et al. \cite{WLY}. However, we did not find in the literature an example of a Riemannian manifold whose sectional curvature changes sign, endowed with an (explicit example of a) strictly convex function. In this work, we construct a Riemannian manifold whose sectional curvature assumes any real value greater than a negative constant and admits a strictly convex function.

%%%%%%%%%%%%%%%%%%%%%%%%%%%%%%%%%%%%%%%%%%%%%%%%%%%%

Finally, we establish a result relating a geometric property of Riemannian manifolds of positive sectional curvature with the set of minimum points of a convex function defined on the manifold. We illustrate, with an example, how this result can be useful to choose the initial point for an iterative method which seeks the minimum of a convex function.

%%%%%%%%%%%%%%%%%%%%%%%%%%%%%%%%%%%%%%%%%%%%%%%%%%%%%

This paper is organized as follows. In Section \ref{Pre} we introduce some notations, basic definitions and important properties of Riemannian manifolds. In Section \ref{conserv} we prove that the conservativity of the geodesic flow of a Riemannian manifold with infinite volume implies that all convex functions on the manifold are constant. In Section \ref{CFeSC} we construct a manifold whose sectional curvature assumes any real value greater than a negative constant and admits a strictly convex function. In Section \ref{SCF} we prove that if the set of minimum points of a convex function, defined on a Riemannian manifold of positive sectional curvature, is not empty then we can minimize it using souls of the manifold. 

\section{Preliminaries\label{Pre}}

In this section, we present some basics from Riemannian geometry. All the manifolds and vector fields here considered will be assumed to be differentiable (smooth). Let $M$ be a (smooth) manifold, we denote the space of (smooth) vector fields over $M$ by $\mathfrak{X}(M)$, the tangent bundle of $M$ will be denoted by $TM$ and the ring of smooth functions over $M$ by $\mathcal{D}(M)$. 

\begin{definition}
	An $r$-covariant tensor field $\omega$ on a Riemannian manifold $M$ is a $\mathcal{D}(M)$-multilinear mapping
	\begin{eqnarray*}
		\omega:&\underbrace{\mathfrak{X}(M)\times\cdots\times\mathfrak{X}(M)}&\to
		\mathcal{D}(M).\\
		&r\,\,\,factors&
	\end{eqnarray*}
\end{definition} 

A couple $(M, \langle\,,\,\rangle)$, where $M$ is a manifold and $\langle\,,\,\rangle$ is a smooth metric,  (inner product on each tangent space varying smoothly on $M$), is called a Riemannian manifold and $\langle\,,\,\rangle$ a Riemannian metric. The metric induces a map $f\in\mathcal{D}(M)\to\nabla\,f\in \mathfrak{X}(M)$ which associates to each $f$, its gradient. Let $D$ be the Levi-Civita connection associated to ($M,\langle\,,\,\rangle$). The differential of $X\in\mathfrak{X}(M)$ is the $\mathcal{D}(M)$-linear operator $A_X:\mathfrak{X}(M)\to\mathfrak{X}(M)$, given by $A_X(Y):=D_YX$. To each point $p\in M$, we assign the linear map $A_X(p):T_pM\to T_pM$ defined by $A_X(p)v=D_vX$. In particular, if $X=\nabla\,f$, then $A_X(p)$ is the Hessian of $f$ at $p$ and is denoted by ${\rm Hess}\,f$.

Given a vector field $V\in\mathfrak{X}(M)$ on a Riemannian manifold $M$ and an $r$-covariant tensor field $\omega$, the Lie derivative of $\omega$ with respect to $V$ is defined by $$(\mathcal{L}_V\omega)(X_1,...,X_r)=V\left(\omega(X_1,...,X_r)\right)-
\sum_{i=1}^r\omega(X_1,...,[V,X_i],...,X_r).$$

\noindent For instance, if $\omega=\langle\,,\rangle$ then $(\mathcal{L}_V\langle,\rangle)(X,Y)=\langle D_XV,Y\rangle+\langle X,D_YV\rangle$. We say that $V\in\mathfrak{X}(M)$ is a conformal vector field if there exists $\phi\in\mathcal{D}(M)$, which is called the conformal factor of $V$, such that
$$\mathcal{L}_V\langle\,,\,\rangle=2\phi\langle\,,\,\rangle.$$

An interesting particular case of a conformal vector field $V$ occurs when $D_XV=\phi X$ for all $X\in\mathfrak{X}(M)$, in this case we say that $V$ is closed.

Now, let $\pi:TM\to M$ be the projection map. For any normal neighbourhood $U$ of a point $p\in M$ there is a canonical map $\tau:\pi^{-1}(U)\to T_pM$ defined as follows: for $Z=(q,w)\in\pi^{-1}(U)$, the image $\tau(Z)$ is obtained by a parallel translation of $w$ along the unique geodesic arc in $U$ joining the point $q=\pi(Z)$ to $p$. The connection map corresponding to $D$ is a map $\kappa:T(TM)\to TM$, inducing for any $Z=(p,z)\in TM$ a linear map of $T_Z(TM)$ into $T_{\pi(Z)}M$, and defined as follows: let $A\in T_Z(TM)$ and $\stackrel{\sim}{Z}: t\to\stackrel{\sim}{Z}(t)$ be a path in $TM$ representing $A$ at $t=0$; then
\[
\kappa(A)=\lim_{t\to0}\frac{\tau(\stackrel{\sim}{Z}(t))-z}{t}.
\]

The induced Riemannian metric $G$ in $TM$ (Riemann-Sasaki metric) is determined by the rule
\[
G(A,B)=\langle d\pi(A),d\pi(B)\rangle+\langle \kappa(A),\kappa(B)\rangle,
\]
where $A,B\in T_Z(TM)$ and $Z\in TM$. For more details about Riemann-Sasaki metric see, for instance, Kowalski \cite{K}.

%%%%%%%%%%%%%%%%%%%%%%%%%%%%%%%%%%%%%%%%%%%%%%%%%%%%%%%%%%%%%%

To define soul of a Riemannian manifold one needs, firstly, the following concepts: a Riemannian submanifold $S$ of $M$ is called totally geodesic if all geodesics in $S$ are also geodesics in $M$ and it is said to be totally convex if for all points $p,q$ in $S$, all geodesics joining $p$ to $q$ are contained in $S$.

%(e.g., in a cylinder $Cyl(k,m)=S^k\times \mathbb{R}^m$, the submanifolds $S^p\times W$, where $W\subset \mathbb{R}^m$ is any affine subspace and $S^p =S^n \cap V^{p+1}$ for $0\leq p\leq k$, where $V^{p+1}$ is any $(p+1)$-dimensional vector subspace of $\mathbb{R}^m$ are totally geodesic, but totally convex only if $p=k$ or $p=0$ and, in this last case, we are considering only a connected component).

\begin{definition}
	Let $M$ be a complete manifold and $S\subset M$ a compact totally convex, totally geodesic submanifold such that $M$ is diffeomorphic to the normal bundle of $S$. The submanifold $S$ is called a soul of $M$.
\end{definition}

In general the soul is not uniquely determined, but any two souls of $M$ are isometric. Gromoll and Meyer \cite{CHE} proved that, when the sectional curvature of $M$ is positive, the soul $S$ of $M$ consists of a single point ($simple\,\,point$) and $M$ is diffeomorphic to an Euclidean space. A point $p\in M$ is said to be simple if there are no geodesic loops in $M$ closed at $p$. In their paper \cite{CHE}, Gromoll and Meyer proved that \textit{the set of simple points in $M$ is open} implying that the set of the souls can not consists of a single point.

\section{Conservativity and non-existence of convex functions\label{conserv}}

Let $M$ be a complete Riemannian manifold and $\theta= (p,v)\in TM$, and denote by $\gamma_{\theta}(t)$ the unique geodesic with initial conditions $\gamma_{\theta}(0)=p$ and $\gamma_{\theta}'(0)=v$. For a given $t\in \Bbb{R}$, we define a diffeomorphism of the tangent bundle $TM$
\[
\varphi_t:TM\to TM
\]
as follows $\varphi_t(\theta)=(\gamma_{\theta}(t),\gamma_{\theta}'(t))$. The family of
diffeomorphism $(\varphi_t)$ is in fact a flow (called \textit{geodesic flow}), that is,
it satisfies $\varphi_{t+s}=\varphi_t\circ\varphi_s$.

Denote by $SM$ the unit tangent bundle of $M$, that is, the subset of $TM$ given by those pairs $\theta=(p,v)$ such that $v$ has norm one. Since geodesics travel with constant speed, we see that $\varphi_t$ leaves $SM$ invariant, that is, given $\theta\in SM$ then for all $t\in\Bbb{R}$ we have $\varphi_t (\theta)\in SM$. So, $(\varphi_t)$ preserves the Liouville measure of the unit tangent bundle. The Liouville measure may be described as follows: every inner product in a finite-dimensional vector space induces a volume element in that space, relative to which the cube spanned by any orthonormal basis has volume $1$. In particular, the Riemannian metrics induces a volume element $dv$ on each tangent space of $M$. Integrating this volume element along $M$, we get a volume measure $dx$ on the manifolds itself. The Liouville measure of $TM$ is given, locally, by the product $\mu=dxdv$. If $m$ denotes the Liouville measure restricted to the unit tangent bundle $SM$, it is known (see, for instance, Paternain \cite{GP}) that $m$ is invariant under the geodesic flow, i.e., the diffeomorphism $\varphi_t:SM\to SM$ is a measure-preserving transformation for all $t\in\Bbb{R}$ [that is $m(B)=m(\varphi_t^{-1}(B))$ for all $B\subset SM$ and $t\in\Bbb{R}$].

%Moreover, its restriction $m$ to the unit tangent bundle $SM$ is invariant under 
%the geodesic flow (see, for instance, Paternain \cite{GP}), i.e., the diffeomorphism
%$\varphi_t:SM\to SM$ is a measure-preserving transformation for all $t\in\Bbb{R}$ 
%[that is $m(B)=m(\varphi_t^{-1}(B))$ for all $B\subset SM$ and $t\in\Bbb{R}$].

\begin{definition}
	The geodesic flow is conservative with respect to the Liouville measure if, given any measurable set $A\subset SM$, for $m$-almost all $\theta\in A$ there exists a sequence $(t_n)_{n\in\Bbb{N}}$ in $\Bbb{R}$ converging to $+\infty$ such that $\varphi_{t_n}(\theta) \in A$ for all $t_n$.
\end{definition}

In the class of Riemannian manifolds with finite volume, from Poincar\'e recurrence theorem it follows that the geodesic flow $\varphi_t:SM\to SM$ is conservative. This property was used by Yau (see \cite{Y}) to generalize the result of Bishop and O'Neill (see \cite{BO}), Yau proved that there is no non-trivial continuous convex function on a complete Riemannian manifold of finite volume. We emphasize that there are Riemannian manifolds of infinite volume whose geodesic flow is conservative (see remark below).  

\begin{remark}
	A hyperbolic surface is a complete two-dimensional Riemannian manifold of constant curvature $-1$. Every such surface has the unit disc as universal cover and can be viewed as $H/\Gamma$, where $H$ is the unit disc equipped with the hyperbolic metric and $\Gamma$ is the covering group of isometries of $H$. Nicholls \cite{N} proved that ``\textit{The geodesic flow on the hyperbolic surface $H/\Gamma$ is conservative and ergodic if and only if the Poincar\'e series of $\Gamma$ diverges at $s=1$}". Such surfaces are called of divergence type. In \cite{H}, Hopf proved that geodesic flows on hyperbolic surfaces of infinite area are either totally dissipative or conservative and ergodic. Thus, surfaces of divergent type with infinite area are examples of Riemannian manifolds with infinite volume whose geodesic flow is conservative. 
\end{remark}

\begin{lemma}\label{recurrentpoint}
	If the geodesic flow $\varphi_t:SM\to SM$ is conservative with respect to the Liouville measure, then for 
	$m$-almost all $\theta \in SM$ there exists a sequence $(t_n)_{n\in\Bbb{N}}$ in $\Bbb{R}$ converging to $+\infty$ such that $\varphi_{t_n}(\theta)\to\theta$.
\end{lemma}
\begin{proof}
	Since $SM$ is a manifold, there is a countable basis $\{V_i\}_{i \in \Bbb{N}}$ for the topology of $SM$ such that $m(V_i) <+\infty$, for every $i$. On the other hand, the geodesic flow is conservative thus, for every $i$, there exists a proper subset $U_i\subset V_i$ with $m(U_i) = m(V_i)$ satisfying: if $\theta \in V_i$ then there exists a sequence $(t_n)_{n\in\Bbb{N}}$ in $\Bbb{R}$ converging to $+\infty$ such that $\varphi_{t_n}(\theta) \in V_i$. Now, note that $m(\stackrel{\sim}{U})=0$ where $\stackrel{\sim}{U} =\bigcup_{i\in \Bbb{N}} (V_i \setminus U_i) $ and if $\theta \in SM \setminus\stackrel{\sim}{U}$ there exists a sequence $(t_n)_{n\in\Bbb{N}}$ in $\Bbb{R}$ converging to $+\infty$ such that $\varphi_{t_n}(\theta)\to\theta$. This concludes the proof. \qed
\end{proof}

\begin{remark}
	Cruz Neto et al. \cite{CNMS} considered the case where $M$ is a complete non-compact Riemannian manifold with finite volume, then $m(SM)<+\infty$ and they used the Poincar\'e recurrence theorem to get a characterization of the $C^1$ monotone vector fields.
\end{remark}

In order to prove our first result we need the following lemma

\begin{lemma}\label{recurrentpoint1}
	If the geodesic flow $\varphi_t:SM\to SM$ is conservative with respect to the Liouville measure, then for $m$-almost all $\theta=(p,v)\in SM$ there are sequences $(t_n)_{n\in\Bbb{N}},(s_n)_{n\in\Bbb{N}}$ in $\Bbb{R}$ converging to $+\infty$ such that $\varphi_{t_n}(\theta)\to\theta$ and $\varphi_{-s_n}(\theta)\to \overline{\theta}=(p,-v)$.
\end{lemma}
\begin{proof}
	Consider the map $T:SM\to SM $ defined by $T(p,v)=(p,-v)$, it is not hard to see that $T$ is an isometry with respect to Sasaki metric whose Riemannian volume coincides with the Liouville measure. Consider the set 
	\[
	\Omega=\{\theta=(p,v)\in SM:\,\exists\,(t_n)_{n\in\Bbb{N}}\subset\Bbb{R}\,\,such\,\,that\,\,t_n\to+\infty\,\, and \,\,\varphi_{t_n}(\theta)\to\theta\}.
	\]
	Since $m(SM\setminus\Omega) = 0$ and $T$ is an isometry, it follows that $m(SM\setminus T(\Omega)) = 0$. Therefore, 
	\[
	m(SM\setminus(\Omega\cap T(\Omega)))=m\left((SM\setminus\Omega)\cup(SM\setminus T(\Omega)) \right)=0. 
	\]
	Note that if $\theta=(p,v)\in\Omega\cap T(\Omega)$ then $\theta=(p,v),\overline{\theta}=
	(p,-v)\in\Omega$, since $\varphi_t(\overline{\theta})=\varphi_{-t}(\theta)$ and we are done. \qed
\end{proof}

A function $f:M\to\Bbb{R}$ on a Riemannian manifold $M$ is convex if its restriction to every geodesic in $M$ is a convex function along the geodesic, i.e., if for every geodesic segment $\gamma:[a,b]\to\Bbb{R}$ and every $t\in[0,1]$,\
\[
f(\gamma((1-t)a+tb))\leq(1-t)f(\gamma(a))+tf(\gamma(b)).
\] 
A convex function $f$ is strictly convex if this inequality is strict whenever $t\in(0,1)$. A convex function is always continuous. If $f$ is smooth, it is known that $f$ is (strictly) convex provided its hessian is positive (definite) semidefinite, or equivalently if $(f\circ\gamma)''\geq0\,(>0)$ for every geodesic $\gamma:I\subset\Bbb{R}\to M$.

\begin{theorem}\label{th2}
	Let $M$ be a connected complete Riemannian manifold. If the geodesic flow $\varphi_t:SM\to SM$ is conservative with respect to the Liouville measure, then all continuous convex functions on $M$ are constant.
\end{theorem}
\begin{proof}
	Suppose that $f:M\to\Bbb{R}$ is a convex function. Note that if $\gamma(t)$ is any geodesic in $M$ with 
	\[
	\lim_{t_n\to+\infty}f(\gamma(t_n))=\lim_{s_n\to-\infty}f(\gamma(s_n))\in\Bbb{R},
	\] 
	for sequences $(t_n),\,(s_n)\subset \Bbb{R}$, then $f$ is constant on the geodesic $\{\gamma(t):\,t\in\Bbb{R}\}$.
	
	As the geodesic flow on $M$ is conservative, then by Lemma \ref{recurrentpoint1} it follows that for $m$-almost all point $\theta=(p,v)\in SM$ there are sequences $(t_n)_{n\in\Bbb{N}}, (s_n)_{n\in\Bbb{N}}    $ in $\Bbb{R}$ converging to $+\infty$ such that $\varphi_{t_n}(\theta)\to\theta$ and $\varphi_{-s_n}(\theta)\to\overline{\theta} =(p,-v)$.
	
	Writing $\varphi_t(\theta)=(\gamma_{\theta}(t),\gamma_{\theta}'(t))$, we obtain that for $m$-almost all $\theta=(p,v)\in SM$ there are sequences $(t_n),(\overline{s}_n) \subset\Bbb{R}$ such that $t_n\to+\infty$, $\overline{s}_n:=-s_n\to-\infty$ and 
	\[
	\lim_{t_n\to+\infty}\gamma_{\theta}(t_n)=p=\lim_{\overline{s}_n\to-\infty}\gamma_{\theta}(\overline{s}_n).
	\]
	
	Then, as we noted in the first paragraph of this proof, $f$ must be constant on $\gamma_{\theta}(t)$. Now, let $\theta =(p,v)$ be any point in $SM$ and $\{\theta_i\}=\{(p_i,v_i)\}$ a sequence converging to $\theta$ where each $\theta_i$ satisfies the above property. Using the continuity of the function $f$ and the fact that the geodesic flow is continuous it follows that $f$ is also constant on $\gamma_{\theta}(t)$. In particular, $f$ is locally constant and by connectedness it follows that $f$ is constant. \qed
\end{proof}

\section{Convex functions and the sectional curvature\label{CFeSC}}

The known examples of strictly convex functions are associated with the sign of the sectional curvature of the Riemannian manifold, for example:
\begin{itemize}
	\item (Theorem 1 (a), Greene and Wu \cite{GW}) If $M$ is a complete non-compact Riemannian manifold of positive sectional curvature, then there exists on $M$ a $C^{\infty}$ Lipschitz continuous strictly convex function such that, for every $\lambda\in\Bbb{R}$, $f^{-1}(]-\infty,\lambda])$ is a compact subset of $M$.
	
	\item (Theorem 4.1, \cite{BO}) Let $M$ be a complete simply connected Riemannian manifold of non positive sectional curvatures $K\leq 0$.
	\begin{enumerate}
		\item[(a)] If $S$ is a closed, totally geodesic submanifold of $M$, then the $C^{\infty}$ function $f_S:M\to\Bbb{R}$ defined by $f_S(x):=d^2(x,S)$ is convex.
		\item[(b)] In (a), if $S$ is a single point $p$, then $f_p$ is strictly convex.
	\end{enumerate}
\end{itemize}

\vspace{.2cm}

However, the sign change of the sectional curvatures does not obstructs the existence of convex functions. In the following example we present a convex function defined on the tangent bundle of a Riemannian manifold.

\begin{example}\label{cfTM} 
	Let $(M,\langle\,,\,\rangle)$ be a complete Riemannian manifold and $(TM,G)$ its tangent bundle, where $G$ is the Riemann-Sasaki metric induced on $TM$. A point in $TM$ is represented by an ordered pair $(p,v)$, where $p\in M$ and $v\in T_pM$.
	
	Let us consider the \textit{kinetic energy} $E:TM\to \Bbb{R}$ defined by $E(p,v)=\langle v,v\rangle$. It is known in the literature (Theorem 3.6, pp 205, Udriste \cite{U}) that $E$ is a $C^{\infty}$ convex function on $(TM,G)$. Furthermore, the kinetic energy is not strictly convex. In fact, given a geodesic $\gamma(t)$ we get that $\beta(t)=(\gamma(t),\gamma'(t))$ is a geodesic in $TM$ and $E(\beta(t))$ is constant implying that the function $E:TM\to\Bbb{R}$ is not strictly convex.
	
	Choosing $ M $ so that the sectional curvature of $TM$ changes its sign, we have a Riemannian manifold whose sectional curvature changes its sign and admits a convex function. \qed
\end{example}

In this context, a question naturally arises:

\vspace{.2cm}

\noindent\textbf{Question.} \textit{Does the sign change of curvature imply the non existence of strictly convex functions?}

In the next theorem we present a geometric condition that ensures the existence of (strictly) convex functions on a particular class of complete non-compact Riemannian manifolds. As a corollary we get the (negative) answer to the raised question. More specifically, 
\begin{theorem}\label{th1}
	Let $M$ be a complete non-compact Riemannian manifold and $V\in\mathfrak{X}(M)$ a closed conformal vector field with conformal factor $\phi$. If $\langle\nabla\phi,V\rangle$ $\geq0$, then the energy function $f=\frac{1}{2}\langle V,V\rangle$ is convex. In addition, if the vector field $V$ and the function $\phi$ never vanishes then $f$ is a strictly convex function.
\end{theorem}
\begin{proof} For every vector fields $X,Y\in\mathfrak{X}(M)$ we have
	\[
	X(f)=\langle D_XV,V\rangle=\langle\phi X,V\rangle,
	\]
	then $\nabla f=\phi V$. Now, note that
	\begin{eqnarray*}
		{\rm Hess}\,f(X,Y)&=&\langle D_X\nabla f,Y\rangle=\langle D_X(\phi V),Y\rangle\\
		&=&\langle X,\nabla \phi\rangle\langle V,Y\rangle+\langle\phi D_XV,Y\rangle\\
		&=&\langle X,\nabla \phi\rangle\langle V,Y\rangle+\phi^2\langle X,Y\rangle.
	\end{eqnarray*}
	Since both ${\rm Hess}\,f$ and the metric are symmetric tensors, we deduce
	\[
	\langle X,\nabla \phi\rangle\langle V,Y\rangle=\langle Y,\nabla \phi\rangle\langle V,X\rangle
	\]
	for all $X,Y\in\mathfrak{X}(M)$. So 
	\[
	|V|^2\cdot\nabla\phi=\langle V,\nabla\phi\rangle\cdot V.
	\]
	Therefore,
	\begin{eqnarray*}
		|V|^2\cdot{\rm Hess}\,f(X,X)&=&\langle V,\nabla\phi\rangle\langle V,X\rangle^2
		+\phi^2|V|^2|X|^2\geq \phi^2|V|^2|X|^2.
	\end{eqnarray*}
	
	Since the set of the points of $M$ where $V$ vanishes is a discrete set (see, for instance, Montiel \cite{Mon}), we conclude the proof of the theorem. \qed
\end{proof}

Now, let us apply Theorem \ref{th1} to construct examples of strictly convex functions on Riemannian manifolds whose sectional curvature assumes any real value greater than a negative constant. We first describe some preliminaries which will be useful to understand the construction.

Let $(B,\langle\,,\,\rangle_B)$ and $(P,\langle\,,\,\rangle_P)$ be Riemannian manifolds and $g>0$ be a positive smooth function on $B$. Set $M=B\times P$ (with the structure of product manifold), and let $\pi_{B}:M\to B$ and $\pi_{P}:M\to P$ denote the canonical projections; we equip $M$ with the {\it warped metric} $\langle\,,\,\rangle$, given by
\[
\langle X,Y\rangle=\langle d\pi_{B}(X),d\pi_{B}(Y)\rangle_B+(g\circ\pi_B)^2\langle d\pi_{P}(X),d\pi_{P}(Y)\rangle_P
\]
and denote the resulting Riemannian space by $M=B\times_{g}P$. We remark that the warped metric will be complete, for any $g$ as above if, and only if $B$ and $P$ are complete.

In this case, $B$ is called the base of $M$ and $P$, the fiber. It is easy to see that the \textit{fibers} $\{b\}\times P=\pi_B^{-1}(b)$ and the \textit{leaves} $B\times\{p\}= \pi_P^{-1}(p)$ are Riemannian submanifolds. Tangent vectors to the leaves are called \textit{horizontal} and tangent vectors to the fibers, \textit{vertical}.

\begin{proposition}\label{ONeill} {\rm[Proposition 42 (5), O'Neill \cite{O}]}
	Let $M=B\times_gP$ be a warped product with Riemannian curvature tensor $R$ and let $R^P$ be the lift to $M$ of the Riemannian curvature tensor of $P$. If $U,V,W\in TM$ are vertical, then
	\[
	R(V,W)U=R^P(V,W)U-\frac{|\nabla\,g|^2}{g^2}\left(\langle V,U\rangle W-\langle W,U\rangle V\right).
	\]
	If $U,V\in TM$ are orthonormal, $K$ and $K^P$ denote the sectional curvature of $M$ and $P$ we get
	\[
	K(U,V)=\frac{1}{g^2}\left(K^P(U,V)-|\nabla\,g|^2\right).
	\]
	%where $\nabla\,g$ denotes the gradient of the function $g$.
\end{proposition} 

An interesting class of spaces furnished with closed conformal vector fields is given by the following subclass of warped product spaces: when $B=I\subset\Bbb{R}$ is an open interval.

If $\partial_t=\nabla\pi_I$ is the standard unit vector field on $I$, then $V=(g\circ\pi_I)\partial_t$ is a closed conformal nowhere vanishing vector field on this manifold, with conformal factor $\phi=g'\circ\pi_I$. Therefore, $\nabla\phi=(g''\circ\pi_I) \partial_t$. In this case, the hypothesis of the Theorem \ref{th1} is equivalent to
\[
0\leq\langle V,\nabla\phi\rangle=\langle(g\circ\pi_I)\partial_t,(g''\circ\pi_I) \partial_t\rangle=(g\circ\pi_I)(g''\circ\pi_I).
\]

Thus, making use of Theorem \ref{th1} we conclude that the function $f:M=I\times_{g}P \to\Bbb{R}$ defined by $f=\frac{1}{2}\langle V,V\rangle$ is convex since that $g''\geq0$.

Now, consider the paraboloid of revolution $P^2=\{(x,y,z):\,z=x^2+y^2\}$, $B=\Bbb{R}$ and $g(t)=e^t$. From Proposition \ref{ONeill} the manifold $M^3=\Bbb{R}\times_{e^t}P^2$ has \textbf{vertical} sectional curvature
\[
K(t,x,y)=\frac{1}{e^{2t}}\left(\frac{4}{(1+4x^2+4y^2)^2}-e^{2t}\right).
\]

$M^3$ is a Riemannian manifold whose sectional curvature assumes any real value greater than $-1$ and, from Theorem \ref{th1}, admits a strictly convex function given that $\phi(t,x,y)=|V(t,x,y)|=e^t\ne0$ for all $t\in\,\Bbb{R}$.

We summarize the above construction with the following corollary.

\begin{corollary}\label{corol1}
	There is Riemannian manifold whose sectional curvature assumes any real value greater than a negative constant and admits a strictly convex function.
\end{corollary}

\section{Soul of a manifold and minimization of convex functions\label{SCF}}

In order to prove our third result we need the following lemma.

\begin{lemma}\label{function} {\rm[Theorem 1 (a),\cite{GW}]}
	If $M$ is a complete non-compact Riemannian manifold of everywhere positive sectional curvature, then there exists on $M$ a $C^{\infty}$ Lipschitz continuous strictly convex function such that, for every $\lambda\in\Bbb{R}$, $f^{-1}(]-\infty,\lambda])$ is a compact subset of $M$.
\end{lemma}

\begin{theorem}\label{th3}
	Let $M$ be a complete non-compact Riemannian manifold of positive sectional curvature and let $u:M\to\Bbb{R}$ be a convex function admitting a point $p_0\in M$ such that $u(p_0)=\inf_{M}u> -\infty$. Then, there exists a sequence $(x_k)$ such that $x_k$ is a soul of $M$, $x_k\to p$ and $u(p)=\inf_{M}u$.
\end{theorem}
\begin{proof} Let $f$ be the function whose existence is assured by the previous lemma. Since $f^{-1}(]-\infty,\lambda])$ is a compact subset of $M$, for every $\lambda\in\Bbb{R}$, we have that $\inf_Mf>-\infty$. Consider the function $g:M\to\Bbb{R}$ defined by $g:=f-\inf_Mf$, then $g\geq0$. We can assume, without loss of generality, that $u\geq0$. For $k\in\Bbb{N}$ fixed, consider the function $h_k:=k\cdot u+g$ and note that for $r\in\Bbb{R}$ we have
	\[
	x\in h_k^{-1}(]-\infty,r])\,\Leftrightarrow\,h_k(x)\leq r\,\Leftrightarrow\,ku(x)+g(x)\leq r.
	\]
	
	Then $x\in h_k^{-1}(]-\infty,r])\,\Rightarrow\,g(x)\leq r\,\Rightarrow\,f(x)\leq r+\inf_Mf$. Hence we conclude that $h_k^{-1}(]-\infty,r])\subset f^{-1}(]-\infty,r+\inf_Mf])$, since $f^{-1}(]-\infty,r+\inf_Mf])$ is compact it follows that $h_k^{-1}(]-\infty,r])$ is also. On the other hand, since $h_k$ is strictly convex there exists a unique point $x_k$ such that $h_k(x_k)=\inf_Mh_k$. Now, fixed $p_0$ such that $u(p_0)=\inf_Mu=0$ we get
	\begin{equation}\label{limxn}
	k\cdot u(x_k)+g(x_k)=h_k(x_k)\leq h_k(p_0)=g(p_0).
	\end{equation}
	
	From the above inequality it follows that $0\leq g(x_k)\leq g(p_0)$, so we get $x_k\in g^{-1}(]-\infty,g(p_0)])=f^{-1}(]-\infty,g(p_0)+\inf_Mf])$. As this is a compact set, we have $x_k\to p\in M$ or we can consider a subsequence of $(x_k)$ if necessary. Therefore, from (\ref{limxn}) we have
	\[
	0\leq u(x_k)\leq\frac{g(p_0)}{k}
	\]
	and thus we obtain $u(p)=\inf_Mu=0$. Since $h_k$ is a strictly convex function we have that the set $\{x_k\}$ is totally convex, from [Theorem 2,\cite{CHE}] it follows that every $x_k$ is a soul of $M$. \qed
\end{proof}

In optimization it is important to establish the convergence of numerical methods to find minimum points for convex functions and it is crucial the initial point in the iterative process. The next result illustrates the region where we should start the iterative process when the Riemannian manifold to be considered is a paraboloid.

\begin{corollary}\label{Corol1Th3}
	Let $u$ be a convex function defined on the paraboloid $P^2=\{(x,y,z):\,z=x^2+y^2\}$ such that the minimum set is non empty, then there exists a minimizer $p=(x_0,y_0,z_0)$ for $u$ such that $z_0 \leq \beta$, where $\beta = \sqrt{\frac{3}{4} \left(1+\mu_{1}^{2}\right)}$ and $\mu_{1} - \arctan \mu_{1} = \frac{\pi}{2}$.
\end{corollary}
\begin{proof}
	Let $h:\Bbb{R}^3\to\Bbb{R}$ be the projection on the third coordinate, i.e., $h(x,y,z)=z$. We recall that a point $p\in P^2$ is {\it simple} if there are no geodesic loops in $M$ closed at $p$. Following the terminology of Ling and Recht \cite{REV}, $p$ is a simple point (a Soul) is equivalent to saying that $p$ is not vertex for any loop. Moreover, in \cite{REV} the authors showed that $p$ is not vertex for any loop if, only if, $h(p) < \beta$.
	
	From Theorem \ref{th3}, there is a souls-sequence ($p_k$) of $P$ such that $p_k\to p$ where $u(p)=\min_Pu$. As a consequence of the result proved by Ling and Recht, we get $h(p_k) < \beta$. Since $h$ is a continuous function, we conclude that $h(p)\leq \beta$. \qed
\end{proof}

\section{Conclusions}\label{conclusion}

In Theorem \ref{th2} we have proved that the conservativity of the geodesic flow on a Riemannian manifold $M$ with infinite volume implies that all convex functions on $M$ are constant, this fact generalizes in a certain sense the result proved by Yau \cite{Y}. In the Theorem \ref{th1} we presented a geometric condition that ensures the existence of (strictly) convex functions on a particular class of complete non-compact Riemannian manifolds. Finally, in the Theorem \ref{th3} we have related the geometry of a Riemannian manifold with the set of minimum points of a convex function defined on the manifold.

\end{document}